\documentclass{amsproc}
\usepackage{amssymb}
\usepackage{amsmath}
\usepackage{amsfonts}
\usepackage{graphicx}

\newtheorem{theorem}{Theorem}[section]

\theoremstyle{definition}

\theoremstyle{remark}

\numberwithin{equation}{section}

\input{tcilatex}

\begin{document}
\title[Not a Quotient]{Is the Dirichlet Space a Quotient of $DA_{n}$?}
\author{Richard Rochberg}
\email{rr@math.wustl.edu}
\date{}
\dedicatory{To Bj\"{o}rn Jawerth, for many good memories.}

\begin{abstract}
We show that the Dirichlet space is not a quotient of the Drury-Arveson
space on the $n-$ball for any finite $n.$ The proof is based a quantitative
comparison of the metrics induced by the Hilbert spaces
\end{abstract}

\maketitle

\section{Statement of the Result}

We will consider reproducing kernel Hilbert spaces, RKHS's, on the balls $%
\mathbb{B}^{n}\subset \mathbb{C}^{n},$ $n=1,2,...,\infty .$ We interpret $%
\mathbb{C}^{\infty }$ as the space $\ell ^{2}(\mathbb{Z}_{+})$ and $\mathbb{B%
}^{\infty }$ as its unit ball.

We are interested in $\mathcal{D}$, the Dirichlet space, which consists of
holomorphic functions $f$ defined on the unit disk $\mathbb{B}^{1}\mathbb{=D}%
,$ $f(z)=\sum a_{n}z^{n},$ normed using $\left\Vert f\right\Vert _{\mathcal{D%
}}^{2}=\sum (n+1)\left\vert a_{n}\right\vert ^{2}.$ The space has a
reproducing kernel, $h_{z},$ for evaluating functions at $z,$ given by%
\begin{eqnarray*}
h_{z}(w) &=&\frac{1}{\bar{z}w}\log \frac{1}{1-\bar{z}w},\text{ \ }zw\neq 0,
\\
&=&1,\qquad \text{\ \ \ }\qquad \text{\ \ \ \ \ }zw=0
\end{eqnarray*}%
We denote the normalized kernels by $\widehat{h_{z}}.$

The other spaces we consider are the Drury-Arveson spaces, $%
DA_{n},1,2,...,\infty .$ The space $DA_{n}$ is the Hilbert space of
holomorphic functions on $\mathbb{B}^{n}$ which is defined by the
reproducing kernel $j_{z}^{(n)}$; for $z,w\in \mathbb{B}^{n},$%
\begin{equation*}
j_{z}^{(n)}(w)=\frac{1}{1-\left\langle w,z\right\rangle }.
\end{equation*}%
Here $\left\langle w,z\right\rangle $ is the standard inner product on $%
\mathbb{C}^{n}.$ We will denote the normalized kernels by $\widehat{%
j_{x}^{(n)}}$ and generally omit "$(n)".$

We are interested in the following

\begin{description}
\item[Question 1] Is there, \textit{for some} \textit{finite }$n,$ a map $%
\Phi :\mathbb{B}^{1}\rightarrow \mathbb{B}^{n},$ and a positive function $%
\lambda $ defined on $\mathbb{B}^{1}$ so that for all $z,w\in \mathbb{B}^{1}$
\begin{equation}
h(w,z)=\lambda (z)\lambda (w)j^{(n)}(\Phi (w),\Phi (z))\text{ \ }?
\label{equation}
\end{equation}
\end{description}

The main result in this paper is that Question 1 has a negative answer.

\section{Background}

The spaces $\mathcal{D}$ and $DA_{n}$ have irreducible complete Pick
kernels, CPK. An introduction to such spaces is in \cite{AgMc} and we will
make free use of the results there. More recent information is in \cite{Sha}.

When the general theory of RKHS with CPK is applied to the space $\mathcal{D}
$ it insures that the variation of Question 1 in which $n$\textit{\ is not
required to be finite} has a positive answer. That result holds in general;
if $K$ is a RKHS of functions on a space $\mathbb{X}$, and if $K$ has a CPK,
then there is a map $\Phi _{\mathbb{X}}:\mathbb{X}\rightarrow \mathbb{B}%
^{n(K)}$ so that the analog of (\ref{equation}) holds. Again, $n(K)=\infty $
may be required.

Given such a $\Phi _{\mathbb{X}}$ we define its range, $\limfunc{Ran}(\Phi _{%
\mathbb{X}}),$ to be the image of $\mathbb{X}$ in $\mathbb{B}^{n(K)}.$ Let $%
\limfunc{Span}(\Phi _{\mathbb{X}}(\mathbb{X))}$ be the closed linear span of
the $DA_{n(K\mathbb{)}}-$kernel functions for $x\in \limfunc{Ran}(\Phi _{%
\mathbb{X}}),$ and let $\limfunc{Van}(\Phi _{\mathbb{X}}(\mathbb{X))}$ be
the space of functions in $DA_{n(K\mathbb{)}}$ which vanish on $\limfunc{Ran}%
(\Phi _{\mathbb{X}}).$ That is:%
\begin{eqnarray*}
\limfunc{Ran}(\Phi _{\mathbb{X}}) &=&\left\{ \Phi _{\mathbb{X}}(x\mathbb{)}%
:x\in \mathbb{X}\right\} \subset \mathbb{B}^{n(K)}, \\
\limfunc{Span}(\Phi _{\mathbb{X}}\mathbb{)} &=&\text{closed span of }\left\{
j_{\zeta }^{n(K)}:\xi \in \limfunc{Ran}(\Phi _{\mathbb{X}})\right\} \subset
DA_{n(K\mathbb{)}}, \\
\limfunc{Van}(\Phi _{\mathbb{X}}(\mathbb{X))} &\mathbb{=}&\left\{ f\in
DA_{n(K\mathbb{)}}:f(\xi )=0\text{ \ }\forall \xi \in \limfunc{Ran}(\Phi _{%
\mathbb{X}})\right\} \subset DA_{n(K\mathbb{)}}.
\end{eqnarray*}%
The map which takes kernel functions for $K$ to kernel functions in $DA_{n(K%
\mathbb{)}},$ $k_{x}\rightarrow j_{\Phi _{\mathbb{X}}(x\mathbb{)}}^{n(K)},$
extends by linearity and continuity to a surjective isometry from $K$ to $%
\limfunc{Span}(\Phi _{\mathbb{X}}\mathbb{)}.$ Also, considering the
definitions we see that $\limfunc{Span}(\Phi _{\mathbb{X}}\mathbb{)}^{\perp
}=\limfunc{Van}(\Phi _{\mathbb{X}}(\mathbb{X))}$. Combining these
observations we have that $K$ is the Hilbert space quotient of $DA_{n(K%
\mathbb{)}}$ by $\limfunc{Van}(\Phi _{\mathbb{X}}(\mathbb{X))}:$ 
\begin{eqnarray*}
K &\approx &DA_{n(K\mathbb{)}}\ominus \limfunc{Van}(\Phi _{\mathbb{X}}(%
\mathbb{X))} \\
&\approx &\limfunc{Van}(\Phi _{\mathbb{X}}(\mathbb{X))}^{\perp } \\
&\approx &\limfunc{Span}(\Phi _{\mathbb{X}}\mathbb{)}
\end{eqnarray*}%
This representation of $K$ as a quotient of $DA_{n(K\mathbb{)}}$ is the
source of the title of this paper.

In this situation it is natural to wonder about the optimal value of $n(K)$
for given $K.$ The author learned of this question a few years ago in
discussions with Ken Davidson and Orr Shalit, and the results here have
their origin in those conversations. These and similar questions have also
been considered by John McCarthy and Orr Shalit \cite{McSh}.

\section{A Reformulation Using the Metric $\protect\delta $}

We will recast this Hilbert space question as one about isometric mappings
between metric spaces.

Suppose $K$ is a RKHS of functions on $\mathbb{X}$ with reproducing kernels $%
\left\{ k_{x}:x\in \mathbb{X}\right\} $ and normalized reproducing kernels $%
\left\{ \widehat{k_{x}}\right\} .$ Define, for all $x,w\in \mathbb{X}$

\begin{equation}
\delta (x,w)=\delta _{K}(x,w)=\sqrt{1-\left\vert \left\langle \widehat{k_{x}}%
,\widehat{k_{w}}\right\rangle \right\vert ^{2}}.  \label{def delta}
\end{equation}%
For any $x\in \mathbb{X}$ let $P_{x}$ be the Hilbert space projection of $K$
onto the span of $k_{x}.$

\begin{proposition}[Coburn \protect\cite{Cob}]
$\delta (z,w)=\left\Vert P_{z}-P_{w}\right\Vert .$ In particular $\delta $
is a metric on $\mathbb{X}.$
\end{proposition}

\begin{proof}
See \cite{Cob} or \cite{ARSW}
\end{proof}

This metric will be our main tool. If $K=\mathcal{D}$ the formula for $%
\delta _{\mathcal{D}}$ does not simplify algebraically. On the other hand,
there are informative algebraic rewritings of the formula for $\delta
_{DA_{n}}$.

We begin with the case $n=1.$ Note that $DA_{1}$ is the classical Hardy
space of the disk. Using the definitions we find that for $z,w\in \mathbb{B}%
^{1}=\mathbb{D}$, 
\begin{equation*}
\delta _{DA_{1}}(z,w)=\sqrt{1-\frac{(1-\left\vert z\right\vert
^{2})(1-\left\vert w\right\vert ^{2})}{\left\vert 1-\bar{z}w\right\vert ^{2}}%
}.
\end{equation*}%
When we use the fundamental identity, for points $x,w\in \mathbb{B}^{1}=%
\mathbb{D}\ $%
\begin{equation}
1-\frac{(1-\left\vert z\right\vert ^{2})(1-\left\vert w\right\vert ^{2})}{%
\left\vert 1-\bar{z}w\right\vert ^{2}}=\left\vert \frac{z-w}{1-\bar{z}w}%
\right\vert ^{2},  \label{FI}
\end{equation}%
we find that $\delta _{H^{2}}=\rho _{1},$ the pseudohyperbolic metric on the
disk;%
\begin{equation*}
\rho _{1}(z,w)=\left\vert \frac{z-w}{1-\bar{z}w}\right\vert .
\end{equation*}%
That metric is characterized by the fact that for any $z\in \mathbb{D},$ $%
\rho _{1}(0,z)=\left\vert z\right\vert $ together with the fact that $\rho
_{1}$ is invariant under holomorphic automorphisms of the disk.

For general $n$ we have something very similar. From the definitions we see
that for $z,w\in \mathbb{B}^{n}$, 
\begin{eqnarray*}
\delta _{DA_{n}}(z,w) &=&\sqrt{1-\left\vert \left\langle \widehat{j_{z}},%
\widehat{j_{w}}\right\rangle \right\vert ^{2}} \\
&=&\sqrt{1-\frac{(1-\left\Vert z\right\Vert ^{2})(1-\left\Vert w\right\Vert
^{2})}{\left\vert 1-\bar{z}w\right\vert ^{2}}}.
\end{eqnarray*}

Although it is less well known, there is also a pseudohyperbolic metric $%
\rho _{n}$ on $\mathbb{B}^{n}.$ For our purposes a good reference on that
metric is \cite{DuWe}. There is an identity for simplifying the expression
for $\delta _{DA_{n}},$ similar to but more complicated than (\ref{FI}).
With it one finds that $\delta _{DA_{n}}=\rho _{n}.$ In analogy with $n=1,$ $%
\rho _{n}$ is characterized by knowing that for any $z\in \mathbb{B}^{n}$ 
\begin{equation}
\rho _{n}(0,z)=\left\Vert z\right\Vert  \label{near boundary}
\end{equation}%
and that $\rho _{n}$ is invariant under holomorphic automorphisms of the
ball.

Although we will not use this fact, we note in passing that the metric space
($\mathbb{B}^{n},\delta _{DA_{n}})=(\mathbb{B}^{n},\rho _{n})$ is a standard
model for complex hyperbolic geometry, \cite{Gol}.

Suppose now that Question 1 has a positive answer, and let $\Phi ,$ $\lambda 
$ be the objects guaranteed by that answer. Using (\ref{equation}) and the
previous discussion of the $\delta $'s and $\rho $'s, we would have, for all 
$z,w\in \mathbb{D}$,%
\begin{eqnarray*}
\delta _{\mathcal{D}}(z,w) &=&\delta _{DA_{n}}(\Phi (z),\Phi (w)) \\
&=&\rho _{n}(\Phi (z),\Phi (w))
\end{eqnarray*}%
(The factors of $\lambda $ all cancel.) Hence, to get a negative answer to
Question 1 it suffices to get a negative answer to the following question:

\begin{description}
\item[Question 2] Is there a finite $n$ for which there is an isometric
mapping $\Phi $ of the metric space ($\mathbb{D},\delta _{\mathcal{D}})$
into the metric space ($\mathbb{B}^{n},\delta _{DA_{n}})?$
\end{description}

We will now give a negative answer to that question.

\begin{theorem}
Question 2, and hence also Question 1, have negative answers.
\end{theorem}

\section{Preliminary Estimates}

We want to estimate $\delta _{\mathcal{D}}(0,z)$ for $z$ near the boundary.
The situation is obviously rotation invariant so without loss of generality
we suppose $z$ is real and positive. For convenience we write $z=1-\sigma $
where $\sigma $ is defined by $2\sigma -\sigma ^{2}=e^{-K}$ for some
positive $K.$ Thus $1-z^{2}=e^{-K}.$ We then have, with $K$ large

\begin{eqnarray}
1-\delta _{\mathcal{D}}^{2}(0,z) &=&\frac{1}{-\log (1-z^{2})}=\frac{1}{K},%
\text{ }  \notag \\
\delta _{\mathcal{D}}(0,z) &=&\sqrt{1-\frac{1}{K}}.
\label{distance to center}
\end{eqnarray}%
Now we estimate the $\delta _{\mathcal{D}}$ distance that $z$ is moved by a
rotation through the small angle $\sigma ;$ that is, we want to estimate $%
\delta _{\mathcal{D}}(z,ze^{i\sigma }).$ From the definitions we have 
\begin{equation*}
1-\delta _{\mathcal{D}}^{2}(z,ze^{i\sigma })=\left\vert \frac{\log \left(
1-z^{2}e^{i\sigma }\right) }{\log (1-z^{2})}\right\vert ^{2}
\end{equation*}%
We have $z^{2}=1-2\sigma +\sigma ^{2}.$ Using the Taylor series for $\log
(1-x)$ and for $\exp x$ we find 
\begin{eqnarray*}
1-\delta _{\mathcal{D}}^{2}(z,ze^{i\sigma }) &=&\left\vert \frac{\log \left[
(2-i)\sigma +\left( 2i-\frac{1}{2}\right) \sigma ^{2}+O(\sigma ^{3})\right] 
}{\log e^{-K}}\right\vert ^{2} \\
&=&\left\vert \frac{\log (2-i)+\log \sigma +O(\sigma )}{-K}\right\vert ^{2}.
\end{eqnarray*}%
We now use the estimate $\log \sigma =-K-\log 2+O(\sigma )$ and continue with

\begin{equation*}
1-\delta _{\mathcal{D}}^{2}(z,ze^{i\sigma })=\left\vert \frac{-K+\log
(1-i/2)+O(\sigma )}{-K}\right\vert ^{2}.
\end{equation*}%
Write $\log (1-i/2)=A+iB$ with $A,B$ real and $A>0.$ Hence, with $O(1)$ and $%
O(\sigma )$ denoting \textit{real} quantities, we have%
\begin{eqnarray}
1-\delta _{\mathcal{D}}^{2}(z,ze^{i\sigma }) &=&\left\vert 1-\frac{%
A+O(\sigma )}{K}+i\frac{O(1)}{K}\right\vert ^{2}.
\label{distance to each other} \\
&=&1-\frac{2A}{K}+\frac{O(1)}{K^{2}}.  \notag
\end{eqnarray}

The invariant ( Poincare-Bergman, hyperbolic) volume of a pseudohyperbolic
ball of radius $r$ is a function of the radius only, not the center. If the
center is selected to be the origin then one can compute the volume
explicitly. We record the formula from \cite[(???)]{DuWe}, if the radius is $%
r$ then the volume is 
\begin{equation}
V(r)=\frac{r^{2n}}{\left( 1-r^{2}\right) ^{n}}.  \label{volume}
\end{equation}%
In particular%
\begin{eqnarray}
V(r) &=&r^{2n}+O(r^{2n+1})\text{ \ \ \ as \ }r\rightarrow 0
\label{asymptotic} \\
&=&2^{n}(1-r)^{-n}+O((1-r)^{-n+1})\text{ \ \ as }r\rightarrow 1  \notag
\end{eqnarray}

\section{Proof of the Theorem}

\begin{proof}
Suppose the answer to Question 2 is positive. Select and fix a finite $n$
and a map $\Phi $ whose existences are insured by that answer.

Fix a large positive $K$ with the property that, with $\sigma $ defined as
above, $N=2\pi /\sigma $ is an integer.

Consider the circle centered at the origin and with radius $1-\sigma .$ On
that circle select $N$ equally spaced points $\left\{ z_{i}\right\} .$ Now
consider the image points $\left\{ \Phi (z_{i})\right\} .$ Because $\Phi $
is an isometry, and noting (\ref{distance to center}), we have 
\begin{equation*}
\delta _{DA}(0,\Phi (z_{i}))=\delta _{\mathcal{D}}(0,z_{i})=\sqrt{1-\frac{1}{%
K}}
\end{equation*}%
Hence all of the $\left\{ \Phi (z_{i})\right\} $ lie on the sphere $S$ in $%
\mathbb{B}^{n}$ centered at the origin and with $\delta _{DA_{n}}$ radius $%
\sqrt{1-1/K}.$ Also, again using the isometry property, and now noting (\ref%
{distance to each other}), we have, for $z_{i}\neq z_{j},$ and for some $B>0$%
\begin{equation*}
\delta _{DA}(\Phi (z_{i}),\Phi (z_{j}))=\delta _{\mathcal{D}}(z_{i},z_{j})>%
\sqrt{B/K}.
\end{equation*}%
Hence if we pick and fix a small $C>0,$ then $\delta _{DA}$ balls $\left\{
B_{i}\right\} $ centered at the points $\left\{ \Phi (z_{i})\right\} $ and
having radius $\sqrt{C/K}$ will be disjoint.

These balls have centers on $S,$ the boundary of the ball $B_{\delta
_{DA}}(0,\sqrt{1-1/K}),$ and hence certainly do not lie inside that ball.
However they do lie inside a slightly large concentric ball whose radius we
now estimate. The metric $\delta _{DA_{n}}$ satisfies a strengthened version
of the triangle inequality, \cite[equ. (**)]{DuWe}. Hence if $\zeta $ is
inside one of the $B_{i}$ then, recalling that we are interested in large $K$%
, we can make the following estimates; the first line is the strengthened
triangle inequality for $\delta _{DA}.$ 
\begin{eqnarray*}
\delta _{DA}(0,\zeta ) &\leq &\frac{\delta _{DA}(0,\Phi (z_{i}))+\delta
_{DA}(\Phi (z_{i}),\zeta )}{1+\delta _{DA}(0,\Phi (z_{i}))\delta _{DA}(\Phi
(z_{i}),\zeta )} \\
&\leq &\frac{\sqrt{1-\frac{1}{K}}+\sqrt{\frac{C^{2}}{K}}}{1+\sqrt{1-\frac{1}{%
K}}\sqrt{\frac{C^{2}}{K}}} \\
&=&1-\left( \frac{1}{2}+C^{2}\right) \frac{1}{K}+O\left( \frac{1}{K^{3/2}}%
\right) \\
&\leq &1-\frac{1}{3K}.
\end{eqnarray*}

Thus we have $N=2\pi /\sigma $ small balls inside $B_{\delta
_{DA}}(0,1-1/\left( 3K\right) ).$ By our estimates for their radius and for
the distance between their centers we see that the small balls are disjoint.
Finally, we have an estimate for the number of them; if $K$ is large then $%
\sigma <2\pi e^{-K}$ and hence $N>e^{K}.$ However, by comparing volumes, we
see that this combination of estimates in impossible.

From (\ref{asymptotic}) we find that, with $A$ and $B$ positive constants
that are independent of $n,$ and also independent of $K$ if $K$ is large, we
have:%
\begin{eqnarray*}
N &=&\text{ number of small balls}\geq e^{K} \\
V_{S} &=&\text{ volume of each small ball }\geq \left( \frac{A}{K}\right)
^{n} \\
V_{L} &=&\text{ volume of large ball }\leq \left( BK\right) ^{n}.
\end{eqnarray*}%
We must have $NV_{S}\leq V_{L}$ for all $K,$ no matter how large; but the
previous estimates show this is impossible, no matter what the values of $%
A,B.$
\end{proof}

\section{Final Remarks}

Although the result just proved and various related results are proved more
directly in \cite{McSh}, this use of $\delta $ provides a different insight
into what is going on.

It is not clear how general the argument is. It does not seem to apply
directly to the Hilbert spaces $\mathcal{D}_{\alpha }$, $0<\alpha <1,$ which
are defined by the kernel functions $(1-\bar{z}w)^{-\alpha }.$ Perhaps this
is not surprising. The space $\mathcal{D}$ is formally a limiting case of
these spaces as $\alpha \rightarrow 0,$ However it is known that the metric $%
\delta _{\mathcal{D}}$ is fundamentally different from $\delta _{\mathcal{D}%
_{\alpha }},$ $0<\alpha <1.$ That difference is discussed in \cite{Roc}. The
theorem is trivially false for $\alpha =1.$

On the other hand the argument seems to give a similar result for the spaces 
$H_{n},$ the HSRKs of functions on the disk defined by the reproducing
kernels%
\begin{eqnarray*}
h_{z}^{(n)}(w) &=&\left( \frac{1}{\bar{z}w}\log \frac{1}{1-\bar{z}w}\right)
^{n},\text{ \ }zw\neq 0, \\
&=&1,\qquad \text{\ \ \ }\qquad \qquad zw=0
\end{eqnarray*}%
for $n=2,3,....$ Those spaces are studied in \cite{AMPRS}.

It seems plausible that there are local, or even infinitesimal, versions of
this argument. Such a result might say that mapping a small $(\mathbb{B}%
^{1}, $ $\delta _{\mathcal{D}})$ neighborhood of $z$ into $(\mathbb{B}^{n},$ 
$\delta _{DA_{n}}),$ even approximately isometrically, is increasingly
difficult, and eventually impossible, as $z$ approaches the boundary.

An infinitesimal version might involve a curvature obstacle. One can pass to
the Riemannian metric which is the infinitesimal version of the $\delta
^{\prime }s.$ Some discussion of this is in \cite{ARSW} and the references
there. Starting from the pseudohyperbolic metric on the ball this produces
the classical Bergman-Poincare metric, the sectional curvatures of which are
always be between two negative constants. There are formulas for the
analogous curvatures related to a general metric $\delta _{K}$ on the disk
but the formulas are daunting. If $k(z,\bar{z})$ is the kernel function then
the Riemannian metric is $\alpha \left\vert dz\right\vert $ with $\alpha
^{2}=\Delta \log k.$ The curvature is then $\kappa =\left( -\Delta \log
\alpha \right) /\alpha ^{2}$.

Finally, it would be interesting to recast these ideas in purely Hilbert
space terms. The metric $\delta $ measures the sine of the angle between
reproducing kernels. Hence the analogs of the "approximate isometries" on
the metric spaces would be linear maps between spans of sets of kernel
functions, where the maps would be subject to an appropriate rigidity
constraint.

\end{document}